\documentclass{amsart}

\usepackage[utf8]{inputenc}
\usepackage{verbatim}
\usepackage{graphicx}
\usepackage{graphicx,caption2,psfrag,float,color}
\usepackage{amssymb}
\usepackage{amscd}
\usepackage{amsmath}

\newtheorem{theorem}{Theorem}[section]

\newtheorem{cor}[theorem]{Corollary}
\newtheorem{lemma}[theorem]{Lemma}

\numberwithin{equation}{subsection}
\newtheorem{definition}[theorem]{Definition}

\title{The order of magnitude for moments for certain cotangent sums}
\author{Helmut Maier and Michael Th. Rassias}
\date{\today}
\address{Department of Mathematics, University of Ulm, Helmholtzstrasse 18, 89081 Ulm, Germany.}
\email{helmut.maier@uni-ulm.de}
\address{Department of Mathematics, ETH-Z\"{u}rich, R\"{a}mistrasse 101, 8092 Z\"{u}rich, Switzerland \& Department of Mathematics, Princeton University, Fine Hall, Washington Road, Princeton, NJ 08544-1000, USA}
\email{michail.rassias@math.ethz.ch, michailrassias@math.princeton.edu}
\thanks{}

\begin{document}

 \maketitle
 
\begin{abstract} 
We settle a question on the rate of growth of the moments of cotangent sums considered by the authors in their previous papers \cite{mr}, \cite{growth}. We even obtain the true order of magnitude of these
moments. We include as well the moments of order $2k$ between fixed multiples of $(2k)!/\pi^{2k}$. \\ \\
\textbf{Key words:} Cotangent sums; equidistribution; moments; continued fractions; measure.\\
\textbf{2000 Mathematics Subject Classification:}  26A12; 11L03; 11M06.%
\newline

\end{abstract}
\section{Introduction}
The authors in joint work and the second author in his thesis, investigated the distribution of cotangent sums
$$c_0\left(\frac{r}{b}\right)=-\sum_{m=1}^{b-1}\frac{m}{b}\cot\left(\frac{\pi m r}{b} \right),$$
as $r$ ranges over the set
$$\{r\::\: (r,b)=1,A_0b\leq r\leq A_1b  \},$$
where $A_0,A_1$ are fixed with $1/2<A_0<A_1<1$ and $b$ tends to infinity.\\
Especially, they considered the moments
$$H_k=\lim_{b\rightarrow+\infty}\phi(b)^{-1}b^{-2k}(A_1-A_0)^{-1}\sum_{\substack{A_0b\leq r\leq A_1b \\ (r,b)=1}}c_0\left(\frac{r}{b} \right)^{2k},\  k\in\mathbb{N},$$
where $\phi(\cdot)$ denotes the Euler phi-function.\\
They could show that all the moments $H_k$ exist and that
$$\lim_{k\rightarrow+\infty}H_k^{1/k}=+\infty.$$
Thus the series $\sum_{k\geq 0}H_kx^{2k}$converges only for $x=0$.\\
It was left open, whether the series
\[
\sum_{k\geq 0}\frac{H_k}{(2k)!}x^k\tag{*}
\]
converges for values of $x$ different from $0$. This fact would considerably simplify the proof for
the distribution of the cotangent sums $c_0(r/b)$ (uniqueness of measures determined by their moments, see \cite{billi}, Section 30, The Method of Moments, Theorem 30.1).\\
Essential for the investigation was the result:
$$H_k=\int_0^1\left(\frac{g(x)}{\pi} \right)^{2k}dx,$$
where
$$g(x)=\sum_{l\geq 1}\frac{1-2\{lx\}}{l}.$$
The function $g$ has been also investigated in the paper \cite{bre} of R. de la Bret\`eche and G. Tenenbaum. More recently it has also been investigated by M. Balazard and B. Martin \cite{balaz2}. We are indebted to M. Balazard for this information. Their ideas, as well as ideas from the paper of S. Marmi, P. Moussa and J. -C. Yoccoz \cite{Marmi} will play an important role in our paper. Independently also S. Bettin \cite{bettin} investigated this problem
and obtained the positivity of the radius of convergence of the series (*). He could also replace the interval $(1/2,1)$ for $A_0,A_1$ by the interval $(0,1)$. We are thankful to S. Bettin
for reading an earlier version of this paper and for providing useful remarks. We shall show the following theorem.
\begin{theorem}\label{x:maint}
There are constants $c_1,c_2>0$, such that
$$c_1\Gamma(2k+1)\leq\int_0^1g(x)^{2k}\:dx\leq c_2\Gamma(2k+1),$$
for all $k\in\mathbb{N}$, where $\Gamma(\cdot)$ stands for the Gamma function.
\end{theorem}
\begin{cor}\label{x:122}
The series 
$$\sum_{k\geq 0}\frac{H_k}{(2k)!}x^k$$
has radius of convergence $\pi^2$.
\end{cor}
%
%
\section{Proof of Theorem \ref{x:maint}}
Let $X=(0,1)\setminus\mathbb{Q}$. This notation shall be used throughout the paper.
\begin{definition}\label{x:protos}
Let $\alpha(x)=\{1/x\}$ for $x\in X.$ The iterates $\alpha_k$ of $\alpha$ are defined by $\alpha_0(x)=x$ and $$\alpha_k(x)=\alpha(\alpha_{k-1}(x)),\ \text{for}\ k>1.$$
\end{definition}
\begin{lemma}
Let $x\in X$ and let 
$$x=[a_0(x);a_1(x),\ldots,a_k(x),\ldots]$$
be the continued fraction expansion of $x$. We define the partial quotient of $p_k(x)$, $q_k(x)$
by
$$\frac{p_k(x)}{q_k(x)}=[a_0(x);a_1(x),\ldots,a_k(x)]\:,\ \text{where}\ (p_k(x),q_k(x))=1\:.$$
Then we have
$$a_k(x)=\left\lfloor  \frac{1}{\alpha_{k-1}(x)} \right\rfloor\:,$$
$$p_{k+1}=a_{k+1}p_k+p_{k-1}\:$$
and
$$q_{k+1}=a_{k+1}q_k+q_{k-1}.$$
\end{lemma}
\begin{proof} (cf. \cite{hens}, p. 7).\end{proof}
\begin{definition}
Let $x\in X$. Let also 
$$\beta_k(x)=\alpha_0(x)\alpha_1(x)\cdots \alpha_k(x)$$
(by convention $\beta_{-1}=1$) and
$$\gamma_k(x)=\beta_{k-1}(x)\log\frac{1}{\alpha_k(x)},\ \text{where}\ k\geq 0,$$
so that $\gamma_0(x)=\log(1/x).$\\
The number $x$ is called a \textit{Wilton number}, if the series
$$\sum_{k\geq 0}(-1)^k\gamma_k(x)$$
converges.\\
\textit{Wilton's function} $\mathcal{W}(x)$ is defined by
$$\mathcal{W}(x)=\sum_{k\geq 0}(-1)^k\gamma_k(x)$$
for each Wilton number $x\in(0,1)$. 
\end{definition}
\begin{lemma}\label{x:lem3}
A number $x\in X$ is a Wilton number if and only if $\alpha(x)$ is a Wilton number. In this case, we have:
$$\mathcal{W}(x)=\log\frac{1}{x}-x\mathcal{W}(\alpha(x)).$$
\end{lemma}
\begin{proof} This is Proposition 8 of \cite{balaz2}.  \end{proof}
\begin{lemma}\label{x:lem4}
There is a bounded function $H\::\:(0,1)\rightarrow\mathbb{R}$, which is continuous in every irrational number, such that
$$g(x)=\mathcal{W}(x)+H(x),$$
almost everywhere.
\end{lemma}
\begin{proof}
In \cite{balaz2}, the function $\phi_1$ is defined by 
$$\phi_1(t)=\sum_{n\geq 1}\frac{B_1(nt)}{n}=\sum_{n\geq 1}\frac{\{nt\}-1/2}{n}\:.$$
Thus, we have
\[
g(x)=-2\phi_1(x).\tag{1}
\]
By Proposition 2 of \cite{balaz2} we obtain
\[
\phi_1(x)=-\frac{1}{2}\mathcal{W}(x)+G(x)\tag{2}
\]
almost everywhere.\\
The function $G$ is bounded and continuous in the set of irrational numbers. The proof of
Lemma \ref{x:lem4} follows from (1) and (2), by the choice \mbox{$H=-2G$}.
\end{proof}
\begin{lemma}\label{x:tl6}
Let $L\in\mathbb{N}$. We have
$$\int_0^{1}g(x)^Ldx= \int_0^{1}\mathcal{W}(x)^L\:dx+\mathcal{R}(L),$$
where
$$\mathcal{R}(L)=\sum_{m=1}^L\binom{L}{m}\int_0^{1}\mathcal{W}(x)^{L-m}H(x)^m\:dx.$$
\end{lemma}
\begin{proof}
This follows from Lemma \ref{x:lem4} and the binomial theorem.
\end{proof}
\begin{definition}\label{x:def4}
Let $p>1$ and $T\::\: L^p\rightarrow L^p$ be defined by 
$$Tf(x)=xf(\alpha(x)).$$
The measure $m$ is defined by
$$m(\mathcal{E})=\frac{1}{\log 2}\int_{\mathcal{E}}\frac{dx}{1+x}\:,$$
where $\mathcal{E}$ is any measurable subset of $(0,1)$.
\end{definition}
\begin{lemma}\label{x:lem6}
Let $p>1$, $n\in\mathbb{N}$.\\
(i) The measure $m$ is invariant with respect to the map $\alpha$, i.e.
$$m(\alpha(\mathcal{E}))=m(\mathcal{E})\:,$$
for all measurable subsets of $\mathcal{E}\subset(0,1)$.\\
(ii) For $f\in L^p$ we have
$$\int_0^1|T^nf(x)|^pdm(x) \leq g^{(n-1)p}\int_0^1|f(x)|^pdm(x),$$
where 
$$g=\frac{\sqrt{5}-1}{2}<1.$$
\end{lemma}
\begin{proof} (i) This result is well-known.\\
(ii) Marmi, Moussa and Yoccoz, consider in their paper \cite{Marmi} a generalized continued 
fraction algorithm, depending on a parameter $\alpha$, which is the usual continued fraction algorithm for the choice $\alpha=1$. The operator $T_\nu$ is defined in (2.5) of \cite{Marmi} and becomes $T$ for $\alpha=1$, $\nu=1$. Then, (ii) is the content of formulas (2.14), (2.15).
\end{proof}
\begin{definition}
For $n\in\mathbb{N}$, $x\in X$ we define
$$\mathcal{L}(x,n)=\sum_{\nu=0}^n(-1)^\nu(T^\nu l)(x),$$
where $l(x)=\log (1/x).$
\end{definition}
\begin{lemma}\label{x:kl1}
For each $L\in\mathbb{N}$, there is an $n_0=n_0(L)\in\mathbb{N}$, such that for $n\geq n_0-1$ we have:
$$\left|\int_0^1\mathcal{L}(x,n)^Ldm(x)-\int_0^1\mathcal{W}(x)^Ldm(x) \right| \leq L^Le^{-2L}.$$
\end{lemma}
\begin{proof}
By Lemma \ref{x:lem3} we have
$$\mathcal{W}(x)+(T\mathcal{W})(x)=l(x)$$
or in obvious notation
$$((1+T)\mathcal{W})(x)=l(x).$$
Therefore
\begin{align*}
\mathcal{L}(x,n)&=\left((1-T\pm\cdots+(-1)^nT^n)l \right)(x)\tag{3}\\
&=\left((1-T\pm\cdots+(-1)^nT^n)(1+T)\right)\mathcal{W}(x)\\
&=\mathcal{W}(x)-(-1)^{n+1}T^{n+1}\mathcal{W}(x).
\end{align*}
By Lemma \ref{x:lem6}, the sequence $(\mathcal{L}(x,n))_{n=1}^{+\infty}$ forms a Cauchy sequence with respect to the $L_p$-norm
$$\|f\|_p:=\left(\int_0^1|f(x)|^pdx  \right)^{1/p}.$$
Therefore, we have
\[
\sup_{n\in\mathbb{N}}\|\mathcal{L}(x,n) \|_p<+\infty,\tag{4}
\]
for all $p>1$.\\
Setting
$$H(x,n,L)=\mathcal{L}(x,n)^{L-1}+\mathcal{L}(x,n)^{L-2}\mathcal{W}(x)+\cdots+\mathcal{W}(x)^{L-1},$$
we have
\[
\mathcal{L}(x,n)^L-\mathcal{W}(x)^L=\left( \mathcal{L}(x,n)-\mathcal{W}(x)\right)H(x,n,L).\tag{5}
\]
By H\"older's inequality and (5), we have
$$\left|\int_0^1\left(\mathcal{L}(x,n)^L-\mathcal{W}(x)^L \right)dx\right|\leq \|\mathcal{L}(x,n)-\mathcal{W}(x) \|_L\| H(x,n,L) \|_{L/(L-1)}.$$
By the triangle inequality for the $L/(L-1)$-norm, we have
$$\| H(x,n,L) \|_{L/(L-1)}\leq L\max_{0\leq m\leq L-1}\| \mathcal{L}(x,n)^{L-m-1}\mathcal{W}(x)^m \|_{L/(L-1)}.$$
By the Cauchy-Schwarz inequality, we have
\begin{align*}
\left( \| \mathcal{L}(x,n)^{L-m-1}\mathcal{W}(x)^m \|_{L/(L-1)}\right)^L&=\left(\int_0^1 \mathcal{L}(x,n)^{(L-m-1)\frac{L}{L-1}}\mathcal{W}(x)^{\frac{mL}{L-1}}dx\right)^{L-1} \\
&\leq \left(\int_0^1  \mathcal{L}(x,n)^{2(L-m-1)\frac{L}{L-1}} dx \right)^{(L-1)/2}\left(\int_0^1  \mathcal{W}(x)^{\frac{2mL}{L-1}}dx  \right)^{(L-1)/2}.\tag{6}
\end{align*}
From (4), (5) and (6), we obtain:
\[
\sup_{n\in\mathbb{N}}\| H(x,n,L) \|_{L/(L-1)}<+\infty\tag{7}
\]
From (3), (5), (7) and Lemma \ref{x:lem6} we get:
$$\lim_{n\rightarrow+\infty}\int_0^1 \left|\mathcal{L}(x,n)^L-\mathcal{W}(x)^L\right| dx=0.$$
The claim of Lemma \ref{x:kl1} follows.
\end{proof}
In the sequel, we always assume that $n\geq n_0(L)$ and that $L\geq L_0$, where $L_0$
is sufficiently large. We also denote by $\mathbb{N}_0$ the set $\mathbb{N}\cup\{0\}.$
\begin{lemma}\label{x:molisegr}
For $m\in\mathbb{N}_0$, $x\in X$, we have:
$$\alpha_m(x)\alpha_{m+1}(x)\leq \frac{1}{2}.$$
\end{lemma}
\begin{proof}
By Definition \ref{x:protos}, we have $\alpha_1(x)<1$. Therefore, for $x\in(0,1/2)$ we have:
\[
x\alpha_1(x)<\frac{1}{2}.\tag{8}
\]
For $x\in\left(\frac{1}{2},1 \right)$ we set $x=\frac{1}{1+u}$, with $0<u<1$. Then, we have
$\alpha_1(x)=u$ and thus 
$$x\alpha_1(x)=\frac{u}{1+u}=:f(u)\:,\ \ \text{say}.$$
Since $f'(u)=(1+u)^{-2}>0$, we have
\[
x\alpha_1(x)\leq f(1)=\frac{1}{2}. \tag{9}
\]
For $m\geq 1$, we have 
$$\alpha_m(x)\alpha_{m+1}(x)=\alpha_m(x)\alpha_1(\alpha_m(x))\leq \frac{1}{2},$$
by (8) and (9).
\end{proof}
\begin{definition}\label{x:simd6}
Let $d,h\in\mathbb{N}_0$, $h\geq 1$, $u,v\in(0,+\infty)$. Then, we define:
$$ \mathcal{J}(d,h,u,v)=\left\{ x\in X\::\: T^dl(x)\geq u\ \ \text{and}\ \ T^{d+h}l(x)\geq v \right\}. $$
\end{definition}
\begin{lemma}\label{x:siml7}
It holds
$$m\left(\mathcal{J}(d,h,u,v) \right)\leq 2\exp\left( -2^{\frac{h-2}{2}}v\exp\left(2^{\frac{d-2}{2}}u \right) \right).$$
\end{lemma}
\begin{proof}
Let $x\in \mathcal{J}(d,h,u,v)$. This implies that 
\[
T^dl(x)=\alpha_0(x)\alpha_1(x)\cdots\alpha_{d-1}(x)l(\alpha_d(x))\geq u\tag{10}
\]
and
\[
T^{d+h}l(x)=\alpha_0(x)\cdots \alpha_d(x)\cdots\alpha_{d+h-1}(x)l(\alpha_{d+h}(x))\geq v.\tag{11}
\]
By partitioning the products in (10) and (11) into parts $\alpha_m(x)\alpha_{m+1}(x)$ and 
possibly one additional factor $\alpha_k(x)$, we obtain from Lemma \ref{x:molisegr} the following:
\[
\alpha_0(x)\alpha_1(x)\cdots\alpha_{d-1}(x)\leq 2^{-\frac{d-2}{2}}\tag{12}
\]
and
\[
\alpha_{d+1}(x)\cdots\alpha_{d+h-1}(x)\leq 2^{-\frac{h-2}{2}}.\tag{13}
\]
From (10) and (12), we have:
$$l(\alpha_d(x))\geq 2^{\frac{d-2}{2}}u$$
and thus 
\[
\alpha_d(x)\leq \exp\left(-2^{\frac{d-2}{2}}u \right).\tag{14}
\]
From (11) and (13) we get:
\[
l(\alpha_{d+h}(x))\geq2^{\frac{h-2}{2}}v\exp\left(2^{\frac{d-2}{2}}u \right). \tag{15}
\]
The claim of Lemma \ref{x:siml7} follows from (15), since 
$$m(\mathcal{E})\leq 2\mu(\mathcal{E}),$$
with $\mu$ being the Lebesgue measure.
\end{proof}
\begin{definition}\label{x:simdef77}
We set $j_0=L-\left\lfloor \frac{L}{100}\right\rfloor$, $C_2=1/400$. For $j\in\mathbb{Z}$,
$j\leq j_0$, we define the intervals:
$$I(L,j)=\left(\exp(-L+j-1),\: \exp(-L+j) \right).$$
For $\nu\in\mathbb{N}_0$, we set
$$a(L,\nu):=\exp(-C_2L+\nu).$$
$$\mathcal{T}(L,j,0)=\left\{ x\in I(L,j)\cap X\::\: |\mathcal{L}(x,n)-l(x)|\leq \exp(-C_2L) \right\},$$
and for $\nu\in\mathbb{N}$, we set
$$\mathcal{T}(L,j,\nu)=\left\{ x\in I(L,j)\cap X\::\:a(L,\nu-1)\leq |\mathcal{L}(x,n)-l(x)|\leq a(L,\nu) \right\}.$$
For $\nu$, $h\in\mathbb{Z}$, $\nu\geq 1$, $h\geq 0$, we set
$$U(L,j,\nu,h)=\left\{ x\in \mathcal{T}(L,j,\nu)\::\: T^hl(x)\geq 2^{-h}a(L,\nu-1)\right\}.$$
\end{definition}
\begin{lemma}\label{x:siml8}
For $\nu\geq 1$, we have 
$$m(\mathcal{T}(L,j,\nu))\leq 3\exp\left(-\frac{1}{200}\exp\left(-C_2L+\nu-1+\frac{1}{2}\left(L-j\right)\right)\right) $$    
\end{lemma}
\begin{proof}
Let $x\in\mathcal{T}(L,j,\nu)$, $\nu\in\mathbb{N}$. From $x\in I(L,j)\cap X$, we have
\[
l(x)= T^0l(x)\geq L-j.\tag{16}
\]
Since
\[
\mathcal{L}(x,n)-l(x)=-Tl(x)+T^1l(x)\mp\cdots+(-1)^{n}T^{n}l(x) \tag{17}
\]
we have
\[
T^hl(x)\geq 2^{-h}a(L,\nu-1),\tag{18}
\]
for at least one $h\geq 1$.\\
From (16), (17), (18) and Definition \ref{x:simdef77}, we have
\[x\in\bigcup_{1\leq h\leq n}\mathcal{J}(0,h,L-j,2^{-h}\exp(-C_2L+\nu-1)). \tag{19}\]
From Lemma \ref{x:siml7} we obtain
\begin{align*}
&m\left(\mathcal{J}(0,h,L-j,2^{-h}\exp(-C_2L+\nu-1))\right)\\
&\leq 2\exp\left( -2^{-(h+1)} \exp\left(\left(-C_2L+\nu-1\right)+2^{\frac{h-2}{2}}\left(L-j\right)  \right)\right).\tag{20}
\end{align*}
From (19) and (20) the claim of Lemma \ref{x:siml8} follows by summation over $h$.
\end{proof}
\begin{definition}\label{x:212121}
We set
\[
x_0=\exp\left(-\left\lfloor \frac{L}{100}\right\rfloor\right).\tag{21}
\]
\end{definition}
\begin{lemma}\label{x:siml99} Let $L\in\mathbb{N}$, then\\
\noindent{(i)}
$$\int_0^1l(x)^Ldx=\Gamma(L+1),$$
\noindent{(ii)}
$$\int_{x_0}^1l(x)^Ldx=O\left(\Gamma(L+1)\exp\left(-\frac{L}{100} \right) \right).$$
\noindent{(iii)} There is a positive constant $C_3>0$, such that\\
$$\int_0^{x_0}|\mathcal{L}(x,n_0)^L-l(x)^L|dm(x)\leq \Gamma(L+1)\exp(-C_3L).$$
\end{lemma}
\begin{proof} Formulas (i) and (ii) are well known.\\
For the proof of (iii), we write
\[
\mathcal{L}(x,n)=l(x)(1+R(x,n)).\tag{22}
\]
Let $j\leq j_0$. Then by Definition \ref{x:simdef77}, for $x\in\mathcal{T}(L,j,\nu)$, we have
$$l(x)\geq L-j$$
and therefore we get
\[
l(x)\geq \frac{L}{200}. \tag{23}
 \]
 By Definition \ref{x:simdef77} we also have
 \[
 |\mathcal{L}(x,n)-l(x)|\leq \exp(-C_2L+\nu).\tag{24}
 \]
 From (23) and (24), we obtain
 \[
 |R(x,n)|\leq\frac{200}{L}\exp(-C_2L+\nu).\tag{25}
 \]
 We distinguish two cases:\\
 \textit{Case 1:} Let $\nu=0$.\\
 From (25) we have
 \[
|R(x,n)|\leq \exp\left(-\frac{C_2}{2}L \right), \tag{26}
 \] 
 $$\int_{\mathcal{T}(L,j,0)}|\mathcal{L}(x,n_0)^L-l(x)^L|dx\leq \int_{\mathcal{T}(L,j,0)} l(x)^L|(1+R(x,n))^L-1|dx.$$
 From (25) and (26), we have
 \[
\int_{\mathcal{T}(L,j,0)}|\mathcal{L}(x,n)^L-l(x)^L|dx\leq\exp\left(-\frac{C_2}{3}L \right)\int_{\mathcal{T}(L,j,0)}l(x)^Ldx  \tag{27}
 \]
\textit{Case 2:} Let $\nu\geq 1$.\\
Because of the fact that
$$L-j\geq \frac{L}{100},$$
we have
$$\max_{x\in I(L,j)} (l(x)^L) \leq e^{200}\min_{x\in I(L,j)} (l(x)^L) $$
and therefore from (25), it follows that
\begin{align*}
&\int_{\mathcal{T}(L,j,\nu)}|\mathcal{L}(x,n)^L-l(x)^L|dx\leq \exp\left(-C_2L+\nu\right)m\left(\mathcal{T}(L,j,\nu)\right)\max_{x\in I(L,j)} (l(x)^L)\tag{28}\\ 
&\leq 3e^{200}\exp\left(-\frac{1}{200}\exp\left(-C_2L+\nu-1+\frac{1}{2}\left(L-j\right)\right)\exp\left(\left(-C_2L+\nu\right)L\right)\right)\min_{x\in I(L,j)} (l(x)^L).
\end{align*}
From (27) and (28), we obtain for $j\leq j_0$, the following
\[
\int_{I(L,j)\cap X}|\mathcal{L}(x,n)^L-l(x)^L|dx \leq \exp\left(-\frac{C_2}{3}L\right)\int_{I(L,j)}l(x)^Ldx.\tag{29}
\]
The result of Lemma \ref{x:siml99} (iii),  now follows from (i), (ii) by summing (29) for $j\leq j_0$. 
\end{proof}
\begin{lemma}\label{x:siml10}
There is a constant $C_4>0$, such that
$$\int_{(1+x_0)^{-1}}^1|\mathcal{L}(x,n)^L-(-1)^Ll(\alpha_1(x))^L|dm(x)\leq \Gamma(L+1)\exp(-C_4L).$$
\end{lemma}
\begin{proof}
We have
\begin{align*}
\mathcal{L}(x,n)&=l(x)-xl(\alpha_1(x))\pm\cdots+(-1)^{n}x\alpha_1(x)\cdots \alpha_{n-1}(x)l(\alpha_{n}(x))\\
&=l(x)-x(l(\alpha_{1}(x))\mp\cdots+(-1)^{n-1}\alpha_1(x)\cdots l(\alpha_{n-1}(\alpha_1(x))))\\
&=l(x)-x\mathcal{L}(\alpha_1(x),n-1).
\end{align*}
For $x\in((1+x_0)^{-1},1)$, we have
$$l(x)=O(\exp(-C_2L))$$
and
$$x=1+O(\exp(-C_2L)).$$
Therefore,
$$\mathcal{L}(x,n)=-\mathcal{L}(\alpha_1(x),n-1)+R(x),$$
where
$$|R(x)|\leq \exp\left(-\frac{C_2}{2}x\right)\max\left( \mathcal{L}(\alpha_1(x),n-1),1\right).$$
Since the map $\alpha$ preserves the measure $m$, we have
$$\int_{(1+x_0)^{-1}}^1\mathcal{L}(\alpha_1(x),n-1)^Ldm(x)=\int_0^{x_0}\mathcal{L}(x,n-1)^Ldm(x)$$
and
$$\int_{(1+x_0)^{-1}}^1l(\alpha_1(x))^Ldm(x)=\int_0^{x_0}l(x)^Ldm(x).$$
The result of Lemma \ref{x:siml10} now follows by applying Lemma \ref{x:siml99} with $n-1$ instead of $n$.
\end{proof}
\begin{definition}\label{x:def217}
Let $\nu\in\mathbb{N}$. We set 
$$b(L,\nu)=\frac{L}{5}\exp\left(\frac{\nu}{100} \right),\  c(L,\nu)=\frac{L}{1000}\exp\left(\frac{\nu}{100} \right)$$ and
$$U(L,\nu)=\left\{x\in\left(x_0,\frac{1}{2}\right)\::\:b(L,\nu-1)\leq |\mathcal{L}(x,n)|\leq b(L,\nu) \right\}.$$
For $h_1,h_2\in\mathbb{N}\::\: 1\leq h_1,h_2\leq n,\:\nu_1,\nu_2\in\mathbb{N}$, set
\begin{align*}
U_1(L,\nu_1,\nu_2,h_1,h_2)=&\bigg\{ x\in\left(x_0,\frac{1}{2}\right)\::\: 2^{-h_1/4}c(L,\nu_1-1)\leq T^{h_1}l(x)\leq 2^{-h_1/4}c(L,\nu_1),\\
& 2^{-h_2/4}c(L,\nu_2-1)\leq T^{h_2}l(x)\leq 2^{-h_2/4}c(L,\nu_2) \bigg\} .
\end{align*}
We set 
$$U_1(L,\nu)=\bigcup_{\substack{\nu_1,\nu_2\geq \nu\\ 1\leq h_1,h_2\leq n}}U_1(L,\nu_1,\nu_2,h_1,h_2),$$
$$U_0(L,\nu)=U(L,\nu)\setminus U_1(L,\nu).$$
\end{definition}
\begin{lemma}\label{x:lem2180}
Let $x\in U_0(L,\nu).$ Then  there is exactly one $h_0\in\mathbb{N}$, $1\leq h_0\leq n$,
such that
\[
l(\alpha_{h_0}(x))\geq 2^{h_0/4}\frac{L}{3} \exp\left(\frac{\nu-1}{100} \right).\tag{30}
\]
We have
\[
m(U_0(L,\nu))\leq \exp\left(-\frac{L}{3}\exp\left(\frac{\nu-1}{100} \right) \right)\tag{31}
\]
\end{lemma}
\begin{proof}
By Definition \ref{x:def217}, we have
\[
\mathcal{L}(x,n)\geq b(L,\nu-1)=\frac{L}{5}\exp\left( \frac{\nu-1}{100}\right).\tag{32}
\]
Then, (32) implies that 
\[T^{h_0}l(x)\geq 2^{-h_0/4}\frac{L}{100}\exp\left(\frac{\nu-1}{100}\right),\tag{33}\]
for at least one $h_0$, $1\leq h_0\leq n$.\\
Since by Definition \ref{x:def217}, $x\not\in U_1(L,\nu_1,\nu_2,h_1,h_2)$, for any pair $(h_1,h_2)$ such that $1\leq h_1,h_2\leq n$, we have
\[
T^hl(x)<2^{-h/4}c(h,\nu-1),
\]
for all $h\neq h_0$.\\
From (32), it follows that
\[
T^{h_0}l(x)\geq \frac{19}{100}L\exp\left(\frac{\nu-1}{100}\right).\tag{34}
\]
We have
$$T^{h_0}l(x)=x\alpha_1(x)\cdots \alpha_{h_0-1}(x)l(\alpha_{h_0}(x)).$$
Because of the fact that $x< 1/2$, it follows from inequality (34) and Lemma \ref{x:molisegr} that
\[l(\alpha_{h_0}(x))\geq 2^{h_0/4}\:\frac{L}{3}\exp\left(\frac{\nu-1}{100}\right)\tag{35}\]
and thus
\[
\alpha_{h_0}(x)\leq \exp\left(-2^{h_0/4}\:\frac{L}{3}\exp\left(\frac{\nu-1}{100}\right)  \right). \tag{36}
\]
Inequality (31) follows from (36) by summing over $h_0$.
\end{proof}
\begin{lemma}
We have
\begin{align*}
\int_{U_0(L,\nu)}|\mathcal{L}(x,n)|^Ldx&\leq \Gamma(L+1)\exp\left(-\frac{L}{100}\exp\left(\frac{\nu-1}{100}\right)\right).\tag{37}
\end{align*}
\end{lemma}
\begin{proof}
From Definition \ref{x:def217}, Lemma \ref{x:lem2180} and (31), we obtain:
\begin{align*}
\int_{U_0(L,\nu)}|\mathcal{L}(x,n)|^Ldx&\leq b(L,\nu)^L\exp\left(-\frac{L}{3}\exp\left(\frac{\nu-1}{100}\right)\right)\\
&\leq \Gamma(L+1)\exp\left(-\frac{L}{100}\exp\left(\frac{\nu-1}{100}\right)\right).
\end{align*}
\end{proof}
\begin{lemma}\label{x:lem220}
For an absolute constant $C_5>0$, we have
$$\int_{U_1(L,\nu)}|\mathcal{L}(x,n)|^Ldx\leq \Gamma(L+1)\exp(-C_5L)\exp\left(\frac{\nu-1}{100}\right).$$
\end{lemma}
\begin{proof}
Let $\nu_1,\nu_2\geq \nu$, $h_1,h_2\in\mathbb{N}$, $1\leq h_1<h_2\leq n$. We apply
Lemma \ref{x:siml7} with $d=h_1$, $d+h=h_2$, $u=2^{-h_1/4}c(L,\nu_1)$, $v=2^{-h_2/4}c(L,\nu_2)$ and obtain:
\begin{align*}
m(U_1(L,\nu_1,\nu_2,h_1,h_2))&\leq \exp(-2^{h_2/8}c(L,\nu_2)\exp(2^{h_1/8}c(L,\nu_1)))\\
&=: d(L,\nu_1,\nu_2,h_1,h_2),\ \ \text{say}.
\end{align*}
For $x\in U_1(L,\nu_1,\nu_2,h_1,h_2)$, by Definition \ref{x:def217} we have
$$|\mathcal{L}(x,n)|\leq b(L,\nu).$$
Thus
$$\int_{U_1(L,\nu_1,\nu_2,h_1,h_2)}|\mathcal{L}(x,n)|^Ldx\leq b(L,\nu)^Ld(l,\nu_1,\nu_2,h_1,h_2).$$
The claim of Lemma \ref{x:lem220} follows by summing over the quadruplets $(\nu_1,\nu_2,h_1,h_2)$ with
\mbox{$\nu_1,\nu_2\geq \nu$}, $h_1,h_2\in\mathbb{N}$.
\end{proof}
\begin{lemma}\label{x:221}
There is a constant $C_6>0$, such that
$$\int_{x_0}^{1/2}|\mathcal{L}(x,n)|^Ldx\leq \Gamma(L+1)\exp(-C_6L).$$
\end{lemma}
\begin{proof}
We have $(x_0,1/2)\cap X=X_1\cup X_2$, where $$X_1=\bigcup_{\nu\in\mathbb{N}} U(L,\nu),\ \ X_2=X\setminus X_1.$$
From Lemmas \ref{x:lem2180} and \ref{x:lem220}, by summing over $\nu$, we obtain
\[
\int_{X_1}|\mathcal{L}(x,n)|^Ldm(x)=O\left(\Gamma(L+1)\exp(-C_7L)\right),\tag{38}
\]
with $C_7>0$.\\
For $x\in X_2$, by Definition \ref{x:def217}, we have
$$|\mathcal{L}(x,n)|\leq \frac{L}{5}.$$
Therefore, since by Stirling's formula $$\Gamma(L+1)\geq L^L\eta^{-L},$$ for all $\eta>e$ if $L$ is sufficiently large, we have
\[
\int_{X_2}|\mathcal{L}(x,n)|^Ldm(x)\leq L^L5^{-L}=O(\Gamma(L+1)\exp(-C_6L)).\tag{39}
\]
The result follows from (38) and (39).
\end{proof}
\begin{definition}\label{x:def2220}
For $\nu\in\mathbb{N}$, let
$$\mathcal{V}(L,\nu)=\left\{ x\in\left(\frac{1}{2},(1+x_0)^{-1}\right)\::\: b(L,\nu-1)\leq |\mathcal{L}(x,n)|\leq b(L,\nu)\right\}.$$
For $h_1,h_2\in\mathbb{N}$ with $1\leq h_1,h_2\leq n$, $\nu_1,\nu_2\in\mathbb{N}$, let
\begin{align*}
\mathcal{V}_1(L,\nu_1,\nu_2,h_1,h_2)&=\bigg\{x\in\left(\frac{1}{2},(1+x_0)^{-1}\right)\::\: 2^{-h_1/4}c(L,\nu_1-1)\leq T^{h_1}l(x)\leq2^{-h_1/4}c(L,\nu_1),\\
&2^{-h_2/4}c(L,\nu_2-1)\leq T^{h_2}l(x)\leq 2^{-h_2/4}c(L,\nu_2)\bigg\}.
\end{align*}
We set
$$\mathcal{V}_1(L,\nu)=\bigcup_{\substack{\nu_1,\nu_2\geq \nu \\ 1\leq h_1<h_2\leq n}}\mathcal{V}_1(L,\nu_1,\nu_2,h_1,h_2) $$
and
$$\mathcal{V}_0(L,\nu)=\mathcal{V}(L,\nu)\setminus \mathcal{V}_1(L,\nu).$$
\end{definition}
\begin{lemma}\label{x:lem2230}
Let $x\in \mathcal{V}_0(L,\nu)$. Then there is exactly one $h_0\in\mathbb{N}$, $2\leq h_0\leq  n$,
such that 
$$l(\alpha_{h_0}(x))\geq 2^{h_0/4}\:\frac{L}{3}\exp\left(\frac{\nu-1}{100} \right).$$
\end{lemma}
\begin{proof}
By Definition \ref{x:def2220}, we have
\[
\mathcal{L}(x,n)\geq b(L,\nu-1)=\frac{L}{5}\exp\left(\frac{\nu-1}{100} \right).\tag{40}
\]
Because of $$l(x)<\log2,\ l(\alpha_1(x))<l(\alpha_1(1+x_0)^{-1})=\left\lfloor\frac{L}{100}\right\rfloor$$ and
$$\mathcal{L}(x,n)=l(x)-T^1l(x)\pm\cdots+(-1)^{n}T^{n}l(x),$$
(40) implies that
$$T^{h_0}l(x)\geq 2^{-h_0/4}\frac{L}{1000}\exp\left(\frac{\nu-1}{100} \right)$$
for at least one $h_0$, $2\leq h_0 \leq n$.\\
Since by Definition \ref{x:def2220}, $x\not\in \mathcal{V}_1(L,\nu_1,\nu_2,h_1,h_2)$ we have
\[
T^hl(x)<2^{-h/4}c(h,\nu-1),\tag{41}
\]
for all $h\neq h_0$. From (40) and (41) it follows that
\[
T^{h_0}l(x)\geq 2^{h_0/4}\:\frac{19}{100}\exp\left(\frac{\nu-1}{100} \right).\tag{42}
\]
We have
$$T^{h_0}l(x)=x\alpha_1(x)\cdots\alpha_{h_0-1}(x)l(\alpha_{h_0}(x)).$$
Because of the fact that $h_0\geq 2$ and $x\alpha_1(x)\leq 1/2$, inequality (42) implies that
$$l(\alpha_{h_0}(x))\geq 2^{h_0/4}\:\frac{L}{3}\exp\left(\frac{\nu-1}{100} \right)$$
and thus
\[
\alpha_{h_0}(x)\leq \exp\left( -2^{-h_0/4}\:\frac{L}{3}\exp\left(\frac{\nu-1}{100} \right) \right).\tag{43}
\]
Lemma \ref{x:lem2230} follows from (43) by summing over $h_0$.
\end{proof}
From Lemma \ref{x:lem2230} we now deduce in a manner completely analogous to the deduction of Lemma \ref{x:221} the following:
\begin{lemma}\label{x:lem240}
$$\int_{1/2}^{(1+x_0)^{-1}}|\mathcal{L}(x,n)|^Ldx\leq \Gamma(L+1)\exp(-C_8L),$$
for a constant $C_8>0$.
\end{lemma}
\begin{lemma}\label{x:2250}
Let $L=2k$, $k\in\mathbb{N}$. There is a constant $C_9>0$, such that
$$\int_0^1\mathcal{L}(x,n)^{2k}dx=2\Gamma(2k+1)\left(1+O\left( \exp(-C_9L)\right) \right).$$
\end{lemma}
\begin{proof}
This follows from Lemmas \ref{x:siml99}, \ref{x:siml10}, \ref{x:221} and \ref{x:lem2230}.
\end{proof}
%
%
%
%
%
%
%
\textit{Proof of Theorem \ref{x:maint}.}\\
From Lemmas \ref{x:kl1} and \ref{x:2250}, it follows that
\[
\int_0^1\mathcal{W}(x)^{2k}dx=2\Gamma(2k+1)(1+O(\exp(-C_7L)).\tag{44}
\]
By Lemma \ref{x:lem4} we have
\[
g(x)=\mathcal{W}(x)+H(x)\tag{45}
\]
with $|H(x)|\leq C^*_0$ for a constant $C^*_0>0$. \\
We set
$$\mathcal{W}_*(x) =\max(|\mathcal{W}(x)|,1)$$
and obtain from (44) and (45):
$$\int_0^1{\mathcal{W}_*}(x)^{2k} dx=2\Gamma(2k+1)(1+O(\exp(-C_7L))),$$
$$g(x)={\mathcal{W}_*(x)}+{H_*(x)},$$
with $|H_*(x)|\leq C_1^*.$ We use the fact that 
$$\|f \|_{2k}=\left( \int_0^1f(x)^{2k}dx\right)^{1/2k}  $$
is a norm and obtain with a suitable constant $C_2^*>0$ that
\begin{align*}
\| g\|_{2k}&\leq \|{\mathcal{W}_*(x)} \|_{2k}+\|{H_*(x)} \|_{2k}\\
&\leq 2^{1/2k}\Gamma(2k+1)^{1/2k}\left(1+\frac{C_2^*}{2^{1/2k}\Gamma(2k+1)^{1/2k}} \right).
\end{align*}
Therefore 
$$\int_0^1g(x)^{2k}dx\leq 2\Gamma(2k+1)\left(1+\frac{C_2^*}{2^{1/2k}\Gamma(2k+1)^{1/2k}} \right)^{2k}
\leq \Gamma(2k+1)e^{3C_2^*}.$$
Analogously
$$\int_0^1g(x)^{2k}dx\geq \Gamma(2k+1)e^{-2C_2^*},$$
which concludes the proof of Theorem \ref{x:maint} for 
$$c_1=e^{-2C_2^*}\ \ \text{and}\ \ c_2=e^{3C_2^*}.$$
\qed
\newline
\textit{Proof of Corollary \ref{x:122}.} This follows from the fact that 
$$\Gamma(2k+1)=(2k)!\:.$$
\qed
\newline
\newline
\newline
\newline
\noindent\textbf{Acknowledgments.} The second author (M. Th. Rassias) expresses his gratitude to Professor E. Kowalski, who proposed to him this inspiring area of research, for providing constructive guidance and for granting him support for Postdoctoral research.
\vspace{5mm}
%
%
%
%
%

%
%
%
%
%
%
%
%
%

\end{document}